\documentclass{article}
\usepackage[utf8]{inputenc}

\usepackage{amssymb}
\usepackage{amsthm}
\usepackage{amsmath,amscd}
\usepackage[mathscr]{euscript}
\usepackage[all]{xy}
\usepackage{lmodern}
\usepackage[T1]{fontenc}
\usepackage[textwidth=14cm,hcentering]{geometry}
\usepackage[colorlinks=true,linkcolor=red,citecolor=blue]{hyperref}
\usepackage{mathtools}
\usepackage{tikz}
\usetikzlibrary{cd}
\usetikzlibrary{arrows,positioning}

\title{An elementary proof of the naturality of the Yoneda embedding}
\author{ Maxime Ramzi}
\date{}

\newtheorem{thm}{Theorem}[section]
\newtheorem{lm}[thm]{Lemma}
\newtheorem{prop}[thm]{Proposition}
\newtheorem{cor}[thm]{Corollary}

\newtheorem*{thm*}{Theorem}

\theoremstyle{definition}
\newtheorem{defn}[thm]{Definition}

\newtheorem{ex}[thm]{Example}

\newtheorem{rmk}[thm]{Remark}

\newtheorem{warn}[thm]{Warning}

\newcommand{\op}{^{\mathrm{op}}}
\newcommand{\cat}{\mathbf}
\newcommand{\Aut}{\mathrm{Aut}}
\newcommand{\Cat}{{\cat{Cat}}}
\newcommand{\on}{\operatorname}
\newcommand{\id}{\mathrm{id}}
\newcommand{\Fun}{\on{Fun}}
\newcommand{\map}{\on{map}}

\newcommand{\Ss}{\cat S}

\newcommand{\PrL}{\cat{Pr}^L}

\newcommand{\Psh}{\cat{Psh}}
\newcommand{\yo}{\text{\usefont{U}{min}{m}{n}\symbol{'210}}}

\DeclareFontFamily{U}{min}{}
\DeclareFontShape{U}{min}{m}{n}{<-> udmj30}{}

\begin{document}

\maketitle
\begin{abstract}
    We give a short proof that the Yoneda embedding is natural for $\infty$-categories, and further prove that the space of natural transformations that are, pointwise, the Yoneda embedding, is contractible. 
\end{abstract}
\section*{Introduction}
In ordinary category theory, the Yoneda embedding, from a category $C$ to its category of presheaves (of sets) is a fundamental object. It is therefore no surprise that in the homotopy coherent context of $\infty$-categories, the corresponding embedding from an $\infty$-category $C$ to its $\infty$-category of presheaves of spaces is just as fundamental. Recall that $\infty$--categories are analogues of categories where hom-sets are replaced with mapping spaces or mapping homotopy types, and where equality of morphisms is replaced by homotopy. This makes them a very powerful tool for handling homotopy-theoretic situations, but it also makes working with them more complicated than with ordinary categories. 

A key difference with ordinary categories, for instance, is that to specify a functor $F$, one needs to specify more than an objects-and-arrows assignment, as one needs to provide \emph{homotopies} $F(f\circ g)\simeq F(f)\circ F(g)$, as well as higher analogues accounting for the fact that composition of arrows in an $\infty$-category is only associative up to (coherent) homotopy. The Yoneda embedding is an example of a functor, and the main interest of this article - while classically easy to define as $c\mapsto \hom_C(-,c)$, the ``formula'' $c\mapsto \map_C(-,c)$ is only an informal description, and even \emph{defining} it properly takes some work (see \cite[Section 5.1.3]{HTT} for a definition, and a proof of a basic version of the Yoneda lemma). 

Accordingly, to specify a natural transformation between functors, one needs to specify more than a collection of transformations with the property that certain squares commute: one needs to specify the homotopies making those squares commute, as well as higher homotopies witnessing the compatibility of these naturality squares with the aforementioned structure of the functors. 

In particular, to understand how the Yoneda embedding varies as the input $\infty$-category $C$ varies (how natural it is), one needs to understand exactly what we mean by the functor ``$C\mapsto \Psh(C)$''.

On the one hand, just as in classical category theory, the $\infty$-category of presheaves (of spaces) over $C$ has a universal property: the Yoneda embedding witnesses it as the free cocompletion of $C$, i.e. the $\infty$-category freely generated under colimits by $C$.  This universal property allows one to define a functor $C\mapsto \mathcal P(C)$ whose value is pointwise given by $\Psh(C)$, and for which the Yoneda embedding is natural \emph{by design}, see \cite[Section 5.3.6]{HTT} for more detail.

On the other hand, one can define the presheaf functor simply as $C\mapsto \Fun(C\op,\Ss)$. From this perspective, this functor is contravariant, but it sends every arrow to a right adjoint, and therefore one can ``take left adjoints'' to obtain a covariant functor (technically, this is done by means of the un/straightening equivalence). On objects, this also has the same value, $\Psh(C)$, and on arrows, it turns out that one can prove that it also has the same value as the one described above. In particular, as functors defined on the \emph{homotopy category}, they are equivalent. Working $1$-categorically, this would be essentially enough, and in fact one can work in this sort of way to prove the statement for $\Fun(C\op,\mathbf{Set})$ and some weak form of naturality (in this $1$-categorical case, the covariant presheaf functor is only a ``pseudo-functor'', a technicality which is taken care of by the very setting for $\infty$-categories). 

However, as we explained above, in $\infty$-categories, knowing how things work for objects and arrows is not enough. It was proved in \cite[Theorem 8.1]{HHLN} that the Yoneda embedding \emph{is} in fact natural for the second functoriality too, and in particular that the two functors $C\mapsto \mathcal P(C)$ and $C\mapsto \Psh(C)$ are equivalent. The proof therein uses some non-trivial technology, and as the authors put it, ``[it] is a lucky accident that
[their] methods answer it''. In fact, right before their actual proof, in their Remark 8.3., they come just short of actually providing a complete and more elementary proof of this naturality. 

The goal of this short note is to fill the gap of this short proof-attempt, and show that there is, in fact, a short and relatively elementary proof. The current presentation of the proof looks somewhat different, but it originated from an attempt essentially equivalent to their remark. 

The key technical input is a result of Barwick and Kan, which says that every $\infty$-category is the ($\infty$-categorical) localization of some poset. With this as input, the proof is remarkably short. 
\section*{Acknowledgements}
I am grateful to Dustin Clausen for originally pointing out this subtlety, and for many helpful conversations about this note. I also want to thank Fabian Hebestreit for his help and for suggesting I add the section on uniqueness, and finally Bastiaan Cnossen for helpful comments on an earlier version of this note and for his simplification of a proof in the appendix. 

This research was supported by the Danish National Research Foundation through the Copenhagen Centre for Geometry and Topology (DNRF151)
\section*{Conventions}
We work with $\infty$-categories as extensively developped by Lurie in \cite{HTT}.

We use $\Ss$ to denote the $\infty$-category of spaces, $\Cat_\infty$ for the $\infty$-category of small $\infty$-categories, $\widehat{\Cat_\infty}$ for the $\infty$-category of possibly large $\infty$-categories, and $\PrL$ the $\infty$-category of presentable $\infty$-categories and left adjoints. 

As usual, $\Delta$ denotes the usual simplex category, which we see as a full subcategory of $\Cat_\infty$. Finally, for an $\infty$-category $C$, $ho(C)$ denotes its homotopy category. 
\section{Preliminaries on the Yoneda embedding}
We first recall the definition of the Yoneda embedding: 
\begin{defn}
Let $C$ be an $\infty$-category. The mapping space functor $\map_C: C\op\times C\to\Ss$ curries over to a functor $\yo: C\to \Fun(C\op,\Ss) =: \Psh(C)$, which we call the Yoneda embedding. 
\end{defn}

The construction $C\mapsto \Fun(C\op,\Ss)$ is a contravariant functor via precomposition which, by unstraightening corresponds to a cartesian fibration $\underline{\Psh}\to \Cat_\infty$. Because the restriction functors have left adjoints, this cartesian fibration is also coCartesian, and we can therefore straighten it covariantly to get a covariant functor $C\mapsto \Psh(C)$.
\begin{defn}\label{defn : pshfunc}
The above construction is our definition of the covariant presheaf functor $\Psh:\Cat_\infty\to\widehat{\Cat_\infty}$. 
\end{defn}

It is not hard to prove that the Yoneda embedding is weakly natural for this functoriality: 
\begin{prop}{\cite[Proposition 5.2.6.3]{HTT}}
For every morphism $f : C\to C'$, there is a commutative diagram: 
\[\begin{tikzcd}
	C & {C'} \\
	{\Psh(C)} & {\Psh(C')}
	\arrow["\yo"', from=1-1, to=2-1]
	\arrow["\yo"', from=1-2, to=2-2]
	\arrow["{f_!}"', from=2-1, to=2-2]
	\arrow["f", from=1-1, to=1-2]
\end{tikzcd}\]
where $f_!$ is left adjoint to $f^*$, restriction along $f\op$. 
\end{prop}
The universal property of the presheaf category also makes it clear that there is a functor $\mathcal P: \Cat_\infty\to \PrL$ taking $C$ to $\Psh(C)$ and $f$ to $f_!$ for which the Yoneda embedding is natural - it is not a priori clear that they agree, and it will be a corollary of what we prove here. 
\section{The proof}
\subsection{A result of Barwick-Kan, and a rigidity result}
We use this subsection to record the following result, originally due to Barwick and Kan (see \cite[Theorem 6.1]{BK1}, specifically point (iv), as well as \cite{BK2}), as well as a consequence thereof - we cite Kerodon for it to be closer to our language: 
\begin{thm}{\cite[Theorem 6.3.7.1 (02MC)]{kerodon}}\label{thm : poset}
For any $\infty$-category $C$, there exists a relative poset $(P,W)$ and a morphism $P\to C$ witnessing $C$ as the localization $P[W^{-1}]$.
\end{thm}
The consequence we wish to highlight, and which is interesting in its own right, is the following: 
\begin{thm}\label{thm : rigid}
Let $f: \Cat_\infty\to\Cat_\infty$ be a functor, and $\alpha$ a natural isomorphism from the identity of $ho(\Cat_\infty)$ to $ho(f): ho(\Cat_\infty)\to ho(\Cat_\infty)$.  

Then there exists a natural transformation $\overline\alpha :\id\to f$ which is pointwise equivalent to $\alpha$, and in particular is a natural equivalence. 
\end{thm}
\begin{rmk}
Note that Toën \cite{toen}, and later Barwick and Schommer-Pries \cite{BSP} have already proved that $\Cat_\infty$ is relatively ``rigid'' : its automorphism space is discrete, and consists of $\{\id, (-)\op\}$. Theorem \ref{thm : rigid} is another instance of the ``rigidity'' of $\Cat_\infty$. 
\end{rmk}
\begin{proof}
First, we observe that the identity of $\Cat_\infty$ is left Kan extended from the full subcategory $\Delta\subset \Cat$ : indeed, by the ``complete Segal space'' presentation of  $\Cat_\infty$, we see that it is a localization of $\Psh(\Delta)$, for which the analogous claim clearly holds - and this claim is stable under accessible localizations.

It follows that to define a morphism $\id_{\Cat_\infty}\to f$, it suffices to define one between the restrictions of these functors to $\Delta$. 

We know that $f$ restricted to $\Delta$ lands in $\Delta$, and we also know that $\Delta\to ho(\Delta)$ is an equivalence. Therefore our isomorphism $\alpha$ actually provides us with a natural equivalence $\id_{\Delta}\simeq f_{\mid\Delta}$, which we use to define a natural transformation $\overline\alpha : \id_{\Cat_\infty}\to f$ - we are reduced to proving that this is pointwise equivalent to $\alpha$. 

We first observe that it is so on posets : let $P$ be a poset, and $p\in P$, viewed as a morphism $\Delta^0\to P$. We can draw the following diagram : 
\[\begin{tikzcd}
	{\Delta^0} & {f(\Delta^0)} & {\Delta^0} \\
	P & {f(P)} & P
	\arrow[from=1-1, to=2-1]
	\arrow[from=1-1, to=1-2]
	\arrow[from=2-1, to=2-2]
	\arrow[from=1-2, to=2-2]
	\arrow["\cong"', from=2-3, to=2-2]
	\arrow["\cong"', from=1-3, to=1-2]
	\arrow[from=1-3, to=2-3]
\end{tikzcd}\]
where the left hand square is the naturality square from the natural transformation $\overline\alpha$, whereas the right hand square is some homotopy witnessing the fact that the natural isomorphism $\alpha$ is natural in the homotopy category. 

Note that the top map is obviously the identity, so that going right-down gives us $p$, while going down-right gives us the image of $p$ under the natural transformation $\overline\alpha_P$. Because morphisms of posets are determined by what they do on objects, this shows that $P\to f(P)\to P$ is the identity, hence by the 2-out-of-3 property, this shows that $P\to f(P)$ is an equivalence, in fact it is (canonically, uniquely) homotopic to $\alpha_P$. 

Now we use the result of Barwick-Kan, Theorem \ref{thm : poset}, i.e. the fact that any $\infty$-category is a localization of a poset to conclude that $\overline\alpha$ is pointwise equivalent to $\alpha$.

Indeed, for an $\infty$-category $C$ and such a localization $P\to C$, to prove that $C\xrightarrow{\overline\alpha_C} f(C)\xrightarrow{\alpha_C^{-1}} C$ is equivalent to the identity, it suffices to prove so after precomposing with $P$. But now, using naturality of $\overline\alpha$, and some homotopy witnessing the naturality of $\alpha$ along $P\to C$, we find that this composite, precomposed with $P\to C$, is exactly the same as the identity precomposed with $P\to C$: 
\[\begin{tikzcd}
	P & {f(P)} & P \\
	C & {f(C)} & C
	\arrow["{\overline\alpha_C}"', from=2-1, to=2-2]
	\arrow["{\alpha_C^{-1}}"', from=2-2, to=2-3]
	\arrow["{\overline\alpha_P}", from=1-1, to=1-2]
	\arrow["{\alpha_P^{-1}}", from=1-2, to=1-3]
	\arrow[from=1-1, to=2-1]
	\arrow[from=1-2, to=2-2]
	\arrow[from=1-3, to=2-3]
\end{tikzcd}\]

The claim now follows. 
\end{proof}
\subsection{Putting things together}
We are now ready to prove our main theorem:
\begin{thm}\label{mainthm}
There is a natural transformation from the inclusion of small $\infty$-categories to large $\infty$-categories to the presheaf functor $\Psh$ (as described in definition \ref{defn : pshfunc}) which, pointwise, is the Yoneda embedding. 

More succintly: the Yoneda embedding is natural. 
\end{thm}
\begin{proof}
Consider the functor $\Psh : \Cat_\infty\to \widehat{\Cat_\infty}$. Given an $\infty$-category $C$, let $C^\yo\subset\Psh(C)$ denote the essential image of the Yoneda embedding. 

By weak naturality of the Yoneda embedding, for any functor $f: C\to D$, the induced functor $\Psh(C)\to \Psh(D)$ sends $C^\yo$ to $D^\yo$. We can therefore define a unique subfunctor $C\mapsto C^\yo$ of $C\mapsto \Psh(C)$ (this is folklore - for a precise proof, see Corollary \ref{cor : subfunctor}\footnote{In the case of full subcategories, the generality of the appendix is not needed: one can give a construction of the subfunctor via the unstraightening equivalence. We included the appendix nonetheless to record the general statement rather than a special case.}). 

There is (by definition) a natural transformation $C^\yo\to \Psh(C)$, so it suffices to show that the Yoneda embedding $C\to C^\yo$ can be made natural. Because $f: C\mapsto C^\yo$ is a functor $\Cat_\infty\to \Cat_\infty$, we can apply Theorem \ref{thm : rigid}. 

Indeed, by weak naturality of the Yoneda embedding, one can view the Yoneda embedding as a natural isomorphism $\id_{ho(\Cat_\infty)}\cong ho(f)$. By Theorem \ref{thm : rigid}, this natural isomorphism can be lifted to a natural equivalence $\id_{\Cat_\infty}\to f$: this is an equivalence $C\simeq C^\yo$, natural in $C$, which, pointwise, is the Yoneda embedding when post-composed with $C^\yo\to \Psh(C)$. 

The claim follows. 
\end{proof}
As a corollary, we find:
\begin{cor}
The functor $\Psh: \Cat_\infty\to \widehat{\Cat_\infty}$ is equivalent to the functor $\mathcal P$ defined via the free cocompletion universal property, cf. \cite[Corollary 5.3.6.10]{HTT}.
\end{cor}
\begin{proof}
The natural transformation $\yo: C\to \Psh(C)$ exhibits $\Psh(C)$ pointwise as the free cocompletion of $C$, and this completely characterizes (partial) left adjoints, in particular $\mathcal P$. 
\end{proof}
 
\subsection{On uniqueness}
We conclude by observing that there is only one way of making the Yoneda embedding natural. In more detail, we prove: 
\begin{cor}
The fiber of $$\map_{\Fun(\Cat_\infty,\widehat{\Cat_\infty})}(\id_{\Cat}, \Psh)\to \map_{\Fun(ho(\Cat_\infty),ho(\widehat{\Cat_\infty}))}(\id_{ho(\Cat_\infty)}, \Psh) $$ over the Yoneda embedding is contractible. 
\end{cor}
\begin{proof}
Consider a natural Yoneda embedding $\yo : \id_{\Cat_\infty}\to\Psh$. 

It induces a commutative square: 
\[\begin{tikzcd}
	{\Aut_{\Fun(\Cat_\infty,\Cat_\infty)}(\id_{\Cat_\infty})} & {\map_{\Fun(\Cat_\infty,\widehat{\Cat_\infty})}(\id_{\Cat_\infty},\Psh)} \\
	{\Aut_{\Fun(ho(\Cat_\infty),ho(\Cat_\infty))}(\id_{ho(\Cat_\infty)})} & {\map_{\Fun(ho(\Cat_\infty),ho(\widehat{\Cat_\infty}))}(\id_{ho(\Cat_\infty)},\Psh)}
	\arrow[from=1-1, to=1-2]
	\arrow[from=2-1, to=2-2]
	\arrow[from=1-1, to=2-1]
	\arrow[from=1-2, to=2-2]
\end{tikzcd}\]

The transformation $\yo$ is pointwise fully faithful, hence it is a monomorphism of functors, and hence both horizontal arrows are inclusions of components. We further claim that it is a pullback square. First, if $y: \id_{\Cat_\infty}\to \Psh$ is another natural transformation which is pointwise the Yoneda embedding, then it must pointwise factor through $\yo$, hence factor through $\yo$ (this is folklore - see Proposition \ref{prop:pwmono} for a proof). Second, a morphism between functors is an equivalence if and only if it is pointwise an equivalence, and this can be detected in the homotopy category. 

Together with the fact that the horizontal maps are inclusions of components, these two facts imply that the square is a pullback. 

It follows that the fiber in question is the same as the fiber of $$\Aut_{\Fun(\Cat_\infty,\Cat_\infty)}(\id_{\Cat_\infty}) \to \Aut_{\Fun(ho(\Cat_\infty),ho(\Cat_\infty))}(\id_{ho(\Cat_\infty)}).$$

The space $\Aut_{\Fun(\Cat_\infty,\Cat_\infty)}(\id_{\Cat})$ is contractible, by \cite[Théorème 6.3]{toen} (see also \cite[Theorem 1.1]{BSP}), and thus this fiber is the loop space of $\Aut_{\Fun(ho(\Cat_\infty),ho(\Cat_\infty)}(\id_{ho(\Cat_\infty)})$ at the identity. 

The $\infty$-category $\Fun(ho(\Cat_\infty), ho(\Cat_\infty))$ is a $1$-category, hence this automorphism space is a set, and thus its loop space at the identity is contractible. This concludes the proof.
\end{proof}
\newpage 
\appendix
\section{On monomorphisms}
In this appendix, we do a few recollections of some folklore results about monomorphisms in $\infty$-categories. The end goal is to prove the following statement :
\begin{prop}
Given a functor $F : I\to C$, and for each object $i\in I$, a monomorphism $g_i \to F(i)$ such that the maps $g_i\to F(i)\xrightarrow{F(f)} F(j)$ factor through $g_j$, for each $f: i\to j$ in $I$, then there is a unique functor $g: I\to C$ with a natural transformation $g\to F$ identifying, for each $i\in I$, $g(i)\to F(i)$ with $g_i\to F(i)$.
\end{prop}
This will be proved as Corollary \ref{cor:monolift}.

 First, for convenience, we recall the following definition: 
\begin{defn}
Let $C$ be an $\infty$-category and $f: x\to y$ a morphism. It is said to be a monomorphism if for any $z\in C$, the map of spaces $\map(z,x)\to \map(z,y)$ is $(-1)$-truncated. 
\end{defn}
\begin{rmk}
Recall that a map of spaces $X\to Y$ is $(-1)$-truncated if and only if all its fibers are empty or contractible, equivalently if one of the following holds: it is fully faithful; the diagonal map $X\to X\times_Y X$ is an equivalence; it is an inclusion of components. 
\end{rmk}
A key example is the following : 
\begin{ex}\label{ex : mainex}
Suppose $f: C\to D$ is a functor between $\infty$-categories. If $f$ is fully faithful, then it is a monomorphism in $\Cat_\infty$.
\end{ex}
\begin{proof}
For any $\infty$-category $E$, $\map(E,-)$ preserves limits. It therefore suffices to show that $C\to C\times_D C$ is an equivalence. 

It is fully faithful because mapping spaces in pullbacks of $\infty$-categories are pullbacks of mapping spaces, and it is essentially surjective because if we have an equivalence $\gamma: f(x)\simeq f(y)$, then by fully faithfulness of $f$, we can lift it to some equivalence $\overline\gamma: x\simeq y$. 
\end{proof}
\begin{warn}
Being fully faithful is stronger than being a monomorphism !
\end{warn}
The key lemma to prove the main result of this appendix is the following :
\begin{lm}
Let $i: A\to B$ be an essentially surjective functor, and $u: K_0\to K_1$ a faithful functor, i.e. a functor which induces inclusions of components on mapping spaces. 

The space of diagonal fillers for any given commutative square as follows is empty or contractible: 
\[\begin{tikzcd}
	A & {K_0} \\
	B & {K_1}
	\arrow[ from=1-1, to=2-1]
	\arrow[from=1-2, to=2-2]
	\arrow["f"', from=2-1, to=2-2]
	\arrow["g", from=1-1, to=1-2]
	\arrow[dashed, from=2-1, to=1-2]
\end{tikzcd}\]
It is non-empty if and only if for every $a, a'\in A$, and every arrow $\alpha: i(a) \to i(a')$, the arrow $f(\alpha)$ is in the image of the map $\map(g(a),g(a'))\to \map(f(i(a)), f(i(a')))$ induced by $u: K_0\to K_1$. 
\end{lm}
\begin{rmk}
More precisely, what we mean by ``the space of diagonal fillers'' is the mapping space from $B$ to $K_0$ in $(\Cat_\infty)_{A//K_1}$. The description we will use of this mapping space is as a space of factorizations in $(\Cat_\infty)_{A/}$, i.e. the fiber of $\map_{A/}(B,K_0)\to \map_{A/}(B,K_1)$ over $f$. 
\end{rmk}
\begin{rmk}
If $K_0\to K_1$ is further a co/Cartesian fibration, then our assumption on mapping spaces amounts to the assumption that its fibers are posets. 
\end{rmk}
\begin{proof}
We first observe that the diagonal map $K_0\to K_0\times_{K_1}K_0$ is $(-1)$-truncated, i.e. a monomorphism. Indeed, mapping spaces in pullbacks of $\infty$-categories are pullbacks of mapping spaces, and so the assumption on $K_0\to K_1$ implies that this diagonal is fully faithful. Example \ref{ex : mainex} shows that fully faithful implies $(-1)$-truncated. 

Now, this already implies that the space of diagonal fillers is $(-1)$-truncated : indeed, the space of diagonal fillers is the fiber of $\map_{A/}(B,K_0)\to \map_{A/}(B,K_1)$ over $f$, so it suffices to show that this map is $(-1)$-truncated. 
But the diagonal of this map is $\map_{A/}(B,-)$ applied to the diagonal $K_0\to K_0\times_{K_1}K_0$, which is fully faithful. So we are reduced to showing that for a fully faithful functor, the space of diagonal fillers is contractible. This is again folklore - for a precise proof, see \cite[Definition 3.8.1, Remark 3.8.2, Proposition 3.8.7, and Corollary 3.9.5]{Martini}\footnote{Martini proves this statement in much greater generality than what we need, but we were not able to locate an earlier reference in the literature. }.

Now, the condition we stated for the existence of a filler is clearly necessary, so let us now focus on sufficiency. 

For sufficiency, we observe that we have a factorization of $i$ as $A\xrightarrow{t} K_0\times_{K_1}B\to B$. Because $A\to B$ is essentially surjective, the restriction of the projection $K_0\times_{K_1}B\to B$ to the essential image of $t$ is also essentially surjective. We now observe that the condition we gave is exactly the condition imposing that this restriction to the image of $t$ be fully faithful. Indeed, for $a,a'\in A$, we have a pullback square : 

\[\begin{tikzcd}
	{\map(t(a),t(a'))} & {\map(i(a),i(a'))} \\
	{\map(g(a),g(a'))} & {\map(f(i(a)),f(i(a')))}
	\arrow[from=1-1, to=2-1]
	\arrow[from=1-2, to=2-2]
	\arrow[from=2-1, to=2-2]
	\arrow[from=1-1, to=1-2]
\end{tikzcd}\]

where the top horizontal map is given by applying the projection to $B$. The bottom map, as it is given by applying $K_0\to K_1$, is an inclusion of components. Therefore, the top map is one as well. Our assumption is precisely that the top map is surjective on $\pi_0$ - it follows that if it is satisfied, the top map is an equivalence, which is exactly saying that the restriction of the projection $K_0\times_{K_1}B\to B$ to the essential image of $t$ is fully faithful. As it is essentially surjective, it is an equivalence, and thus we get an inverse $B\to K_0\times_{K_1}B$, which is easily seen to give a diagonal filler. 
\end{proof}

\begin{cor}\label{cor:monolift}
Let $C$ be an $\infty$-category, and let $C_{mono}\subset C^{\Delta^1}$ be the full subcategory of the arrow category spanned by monomorphisms. Let $I$ be an arbitrary $\infty$-category, with an essential surjection $p: I_0\to I$, e.g. from a set. Then the space of diagonal fillers in any commutative diagram as below is empty or contractible: 
\[\begin{tikzcd}
	{I_0} & {C_{mono}} \\
	I & C
	\arrow[ from=1-1, to=2-1]
	\arrow[from=1-2, to=2-2]
	\arrow["f"', from=2-1, to=2-2]
	\arrow["m", from=1-1, to=1-2]
	\arrow[dashed, from=2-1, to=1-2]
\end{tikzcd}\] 
where the map $C_{mono}\to C$ is evaluation at $1$, that picks out the target of an arrow. 

It is non-empty if and only if for every arrow $p(i)\to p(j)$ in $I$, the corresponding map $f(p(i))\to f(p(j))$ is compatible with the monomorphisms $m(i) = (h_i\to f(p(i)))$, $m(j) = (h_j\to f(p(j)))$, i.e. the composite $h_i\to f(p(i))\to f(p(j))$ factors through $h_j$. 
\end{cor}
\begin{proof}
By the previous lemma, it suffices to prove that $C_{mono}\to C$ induces inclusions of components on mapping spaces. 

Let $f: x_0\to x_1, g: y_0\to y_1$ be monomorphisms. We have a pullback square: 
\[\begin{tikzcd}
	{\map(f,g)} & {\map(x_1,y_1)} \\
	{\map(x_0,y_0)} & {\map(x_0,y_1)}
	\arrow[from=1-1, to=2-1]
	\arrow[from=1-2, to=2-2]
	\arrow[from=2-1, to=2-2]
	\arrow[from=1-1, to=1-2]
	\arrow["\lrcorner"{anchor=center, pos=0.125}, draw=none, from=1-1, to=2-2]
\end{tikzcd}\]

Because $y_0\to y_1$ is a monomorphism, the bottom map is an inclusion of components, therefore so is the top map, which proves the claim. 
\end{proof}
\begin{cor}\label{cor : subfunctor}
Let $I$ be an $\infty$-category and $f: I \to \Cat_\infty$ be a functor, and let $p: I_0\to I$ be an essentially surjective map from a set. 

Suppose given, for each $i\in I_0$, a fully faithful embedding $g(i)\to f(p(i))$. 

If for every $i,i'\in I_0$ and every map $h: p(i)\to p(i')$ in $I$, the composite $g(i)\to f(p(i)) \to f(p(i'))$ factors through $g(i')$, then there exists an essentially unique lift $\tilde f$ of $f$ along $\Fun(\Delta^1,\Cat_\infty)\xrightarrow{\mathrm{target}}\Cat_\infty$ such that for all $i\in I_0$, $\tilde f(p(i))=  (g(i)\to f(p(i)))$.
\end{cor}

We conclude with another folklore result about monomorphisms: 
\begin{prop}\label{prop:pwmono}
Let $f,g,h: I\to C$ be functors between $\infty$-categories, and $\eta: f\to g$ a transformation which is pointwise a monomorphism.  

In this case, $\eta$ is a monomorphism, and a transformation $\alpha: h\to g$ factors through $\eta$ if and only if for all $i\in I$, $\alpha_i : h(i)\to g(i)$ factors through $\eta_i : f(i)\to g(i)$. 
\end{prop}
We will use the following lemma: 
\begin{lm}\label{lm:factmono}
Let $D$ be an $\infty$-category with finite limits, and $i: x\to y$ a monomorphism. Let $z\in D$ and $p: z\to y$ be a morphism. 

In this situation, $p$ factors through $i$ if and only if the projection map $z\times_y x\to z$ is an equivalence. 
\end{lm}
\begin{proof}
By definition, we have a commutative diagram : 
\[\begin{tikzcd}
	{z\times_y x} & x \\
	z & y
	\arrow["i", from=1-2, to=2-2]
	\arrow["p"', from=2-1, to=2-2]
	\arrow[from=1-1, to=2-1]
	\arrow[from=1-1, to=1-2]
\end{tikzcd}\]

The left vertical leg is a monomorphism, as it is pulled back from a monomorphism. It follows that it is an equivalence, if and only if it has a section. 

But a section of this map is exactly a factorization of $p$ through $i$, so the left vertical leg is an equivalence if and only if there exists such a factorization. 
\end{proof}
\begin{proof}[Proof of Proposition \ref{prop:pwmono}]
Up to embedding $C$ in its presheaf category, we may assume $C$ has finite limits, and hence also the functor category $\Fun(I,C)$, and the latter are computed pointwise. 

In particular, the diagonal map $f\to f\times_g f$ is an equivalence because it is so pointwise, which proves that $f\to g$ is a monomorphism. 

For the second part of the statement, form the pullback $h\times_g f \to h$. 

Evaluating this map at any $i\in I$, Lemma \ref{lm:factmono} shows that this is an equivalence at every $i\in I$, hence an equivalence. Using Lemma \ref{lm:factmono} again, but now in $\Fun(I,C)$, we deduce that $h\to g$ factors through $\eta$. 
\end{proof}
\bibliographystyle{alpha}
\bibliography{Biblio.bib}

\textsc{Institut for Matematiske Fag, K\o benhavns Universitet, Danmark}

\textit{Email adress : }\texttt{maxime.ramzi@math.ku.dk}
\end{document}